\documentclass[a4paper, reqno, 12pt]{amsart}
\usepackage[english]{babel}
\usepackage[T1]{fontenc}
\usepackage{mathptmx,amsmath,dsfont}
\usepackage{hyperref}
\usepackage{graphicx}
\usepackage{amssymb,verbatim,cases,bbold}
\usepackage{fullpage}
\usepackage[svgnames,hyperref]{xcolor}
\newtheorem{theorem}{Theorem}[section]

\newtheorem{lemma}[theorem]{Lemma}
\newtheorem{definition}[theorem]{Definition}

\theoremstyle{remark}

\newtheorem{remark}[theorem]{Remark}
\newtheorem{example}[theorem]{Example}
\numberwithin{equation}{section}

\newcommand{\R}{\mathbb{R}}

\newcommand{\bN}{\mathbb{N}}

\newcommand{\E}{\mathcal{E}}

\newcommand{\QP}{QPSH}

\newcommand{\n}{\noindent}

\newcommand{\norm}[1]{\|{#1}\|}

\begin{document}
	\author{Thai Duong Do\textit{$^{1}$}, Van Thien Nguyen\textit{$^{2}$}}
	\address{\textit{$^{1}$}Department of Mathematics, National University of Singapore - 10, Lower Kent Ridge Road - Singapore 119076\footnote[3]{On leave from Institute of Mathematics, Vietnam Academy of Science and Technology}}
	\email{duongdothai.vn@gmail.com}
	\address{\textit{$^{2}$}FPT University, Education zone, Hoa Lac high tech park, Km29 Thang Long highway, Thach That ward, Hanoi, Viet Nam}
	\email{thiennv15@fe.edu.vn}
	
	\title{On the finite energy classes of quaternionic  plurisubharmonic functions}
	
	\date{\today}
	\maketitle
	\begin{abstract}
		We investigate the finite $p$-energy classes $\E_p$ of quaternionic  plurisubharmonic functions of Cegrell type. We also construct an example to show that the optimal constant in the energy estimate is strictly bigger than $1$ for $p>0$, $p\neq 1$. This leads us to the fact that we can not use the variational method to solve the quaternionic Monge–Amp\`ere equation for the classes $\E_p$ when $p>0$, $p\neq 1$.
	\end{abstract} 
	\tableofcontents
	
	\section{Introduction}
	The notion of quaternionic  plurisubharmonic (abbreviated qpsh) function of quaternionic variables was introduced by S. Alasker in \cite{alesker03a}. Let $\Omega$ be a domain in the flat space $\mathbb{H}^n$. An upper semi-continuous function $u: \Omega \to \mathbb{R}$ is qpsh if $u$ restricted to any affine right quaternionic line intersected with $\Omega$ is subharmonic. For a $C^2$ qpsh function $u=u(q_1, \cdots, q_n)\colon \mathbb{H}^n \to \mathbb{R}$, the quaternionic  Monge-Amp\`ere operators can be defined as the Moore determinant of the quaternionic Hessian of $u$:
	\begin{align*}
		\mbox{det}\begin{bmatrix}
			\frac{\partial^2 u}{\partial q_j\partial \bar{q}_k}
		\end{bmatrix}.
	\end{align*}
	
	Inspirit of \cite{bt76}, Alesker showed that the quaternionic Monge–Ampère operator can be extended to continuous qpsh functions (see in \cite{alesker03b}). Moreover, he proved one of the most powerful tool in the pluripotential theory, the comparison principle version for qpsh functions. In quaternionic manifolds setting, Alesker in \cite{alesker12} introduced an operator in terms of the Baston operator $\Delta$. As the manifold is flat, he showed that the $n$th-power of this operator coincides the quaternionic Monge–Amp\`ere operator. Motivated by Alesker, Wan and Kang in \cite{WK17} introduced two first-order differential operators $d_0$ and $d_1$. The actings of $d_0, d_1$ and $\Delta=d_0d_1$ on differential forms are very similar $\partial, \bar{\partial}$ and $\Delta=\partial\bar{\partial}$ in complex case. Then we can express the quaternionic Monge–Amp\`ere operator as $(\Delta u)^n$. Following Bedford and Taylor in \cite{bt82}, Wan and Kang showed that the product of currents $\Delta u_1 \wedge \cdots \Delta u_k$ associated to the
	locally bounded qpsh functions $u_1, \dots, u_k$ is a closed current. 
	
	The Dirichlet problem for the quaternionic Monge-Amp\`ere equation is written under the general form:
	\begin{align*}
		\begin{cases}
			u\in QPSH(\Omega)\cap C^0(\bar{\Omega})\\
			\left(\Delta u\right)^n = fdV\\
			u\mid_{\partial \Omega}=\phi.
		\end{cases}
	\end{align*}
	Alesker \cite{alesker03b} solved the Dirichlet problem with a continuous boundary data $\phi$ and the right-hand side $f$ continuous up to the boundary on the unit ball. Following the approach of 
	Ko\l odziej in \cite{ko96,ko98}, Sroka \cite{sroka20a} proved the existence of the weak solutions to the Dirichlet problem with the density $f\in L^p(\Omega)$ for $p>2$ and a continuous boundary data $\phi$. Moreover, the exponent $p$ is optimal.
	
	Very recently, the pluripotential theory associated to qpsh functions has been developed by Alesker, Wan, Sroka,...(see \cite{alesker12, sroka20a, sroka20b, W17, W19, W20, W20b, WK17, WW17, WZ15}). A natural question that arises is to find the domain for the quaternionic Monge-Amp\`ere operators.
	Inspired by Cegrell \cite{cegrell04}, Wan introduced Cegrell classes for qpsh functions. She showed that the quaternionic Monge-Amp\`ere operator is well defined on a certain subset $\mathcal{E}(\Omega)$ of non-positive qpsh functions (see \cite{W20}). Following the variational method in \cite{ACC12, bbgz13}, Wan showed that the solution of the Dirichlet problem on the Cegrell finite energy class $\mathcal{E}_1(\Omega)$ minimizing an energy functional.
	
	In the present paper, we investigate the Cegrell's finite $p$-energy classes $\mathcal{E}_p(\Omega)$ without restriction on $p$. In plurisubharmonic cases, they turn out to be a very effective tool in K\"ahler geometry. For these applications, we refer the readers to the papers \cite{begz10, darvas15, darvas17, dgl21, dinew09, gt23}. The finite $p$-energy classes originally were investigated on $\mathbb{C}^n$ in \cite{ACH07, cegrell04} and on K\"ahler manifolds in \cite{gz07}. We shall rely on the techniques presented in these papers to develop  the finite $p$-energy classes on $\mathbb{H}^n$. In details, we shall prove a sequence of theorems for this class such as the comparison principles, the convergence theorems, especially the $p$-energy estimate inequality:
	\vspace{2mm}
	
	\noindent \textbf{Theorem \ref{thm:energyine}.}
	\textit{Let $u_0,...,u_n\in\E_p(\Omega)$. Then
		$$e_p(u_0,...,u_n)\leq D_p e_p(u_0)^{\frac{p}{n+p}}e_p(u_1)^{\frac{1}{n+p}}...e_p(u_n)^{\frac{1}{n+p}},$$
		where $$D_p=\begin{cases}
			p^{\frac{-\alpha(p,n)}{1-p}},& \text{ if }0<p<1,\\
			1,& \text{ if }p=1,\\
			p^{\frac{p\alpha(p,n)}{p-1}},& \text{ if }p>1,
		\end{cases}$$
		and $$\alpha(p,n)=(p+2)\Big( \frac{p+1}{p}\Big)^{n-1}-(p+1).$$}
	\vspace{2mm}\\
	A natural question that arises is to find the optimal constant $C_p$ such that
	\begin{equation}\label{optimal}
		e_p(u_0,...,u_n)\leq C_p e_p(u_0)^{\frac{p}{n+p}}e_p(u_1)^{\frac{1}{n+p}}...e_p(u_n)^{\frac{1}{n+p}},\ \forall u_0,...,u_n\in\E_p(\Omega).
	\end{equation}
	When $p=1$, the optimal constant $C_1$ on the right-hand side in inequality~(\ref{optimal}) is $1$, according to Theorem \ref{thm:energyine}. However, when $p\neq 1$, it remains unknown to us.
	We would like to emphasize that the optimal constant on the right-hand side in inequality~(\ref{optimal}) plays an important role in using the variational approach to solve the quaternionic Monge–Amp\`ere equation when the right-hand side is a positive measure of finite energy (see \cite[Section 4]{W20}). To be more specific, if this constant exceeds $1$, the application of the variational approach to address the quaternionic Monge–Amp\`ere equation is not possible. Regarding this topic, in plurisubharmonic case, \AA hag and Czy\.z have constructed a counterexample based on using Beta function, Gamma function and Digamma function to confirm that this optimal constant must exceed $1$ in the classes $\mathcal{E}_p$ when $p\neq 1$ (see \cite{AC09}). A similar counterexample has also been constructed by Czy\.z and Nguyen in $m$-subharmonic setting (see \cite{CN17}).
	Motivated by their counterexamples, in Section \ref{sec4}, we also construct a counterexample which leads us to the fact that the optimal constant must also exceed $1$ in the classes $\mathcal{E}_p$, $p\neq 1$ in qpsh setting. Due to the complexities in quaternionic analysis, particularly in Moore determinant calculations, the computations in our counterexample differ from those constructed in plurisubharmonic and $m$-subharmonic cases.
	\section{Preliminaries}
	Denote by $QPSH(\Omega)$ the class of qpsh function  on $\Omega$ and by
	$QPSH^{-}(\Omega)$ the subclass of negative qpsh functions. Throughout the paper, $\Omega$ will always be a quaternionic hyperconvex domain of $\mathbb{H}^n$, which means that it is a bounded domain and that there exists a negative qpsh function $\rho$ such that the set $\{q \in \Omega \colon \rho(q) < c\}$ is a relatively compact subset of $\Omega$, for
	any $c<0$. Such function $\rho$ is called an exhaustion function defining $\Omega$. Let $p>0$.
	\begin{definition}
		The quaternionic Cegrell classes are defined as follows    \begin{align*}
			\mathcal{E}_{0}(\Omega)&=\left\{u\in QPSH^-(\Omega)\cap L^\infty(\Omega) \colon \lim_{\Omega\ni z\to\xi_0}u(z)=0 \ \forall \xi_o\in\partial\Omega \ \mbox{and} \  \int\limits_\Omega (\Delta u)^n < \infty\right\},\\
			\mathcal{E}_p(\Omega) &=\left\{u\in QPSH^-(\Omega) \colon \exists \ \{u^j\}\subset \mathcal{E}_0(\Omega),\  u^j\searrow u \text{ and } \sup_j\int\limits_\Omega (-u^j)^p(\Delta u^j)^n<\infty\right\},\\
			\mathcal{F}(\Omega)&=\left\{ u\in QPSH^-(\Omega) \colon \exists \ \{u^j\}\subset \mathcal{E}_0(\Omega),\  u^j\searrow u \text{ and } \sup_j \int\limits_\Omega (\Delta u^j)^n < \infty \right\},\\
			\mathcal{E}(\Omega)&=\left\{ u\in QPSH^-(\Omega) \colon \text{ for every } K\Subset \Omega,\ \exists \ u_K\in \mathcal{F}(\Omega) \text{ such that } u_K=u \text{ on } K  \right\}.
		\end{align*}
	\end{definition}
	The quaternionic Monge–Amp\`ere operator is well-defined for functions in the class $\E(\Omega)$ as follows.
	\begin{theorem}\label{converges e}(\cite[Theorem  3.2]{W20})
		Let $u_1,...,u_n\in\E(\Omega)$ and $(u_1^j),...,(u_n^j)\subset\E_0(\Omega)$ such that $u_k^j\searrow u_k$ for every $k=1,...,n$. Then the sequence of measures $\triangle u_1^j\wedge...\wedge\triangle u_n^j$ converges weakly to a positive Radon measure which does not depend on the choice of the sequences $(u_k^j)$. We then define $\triangle u_1\wedge...\wedge\triangle u_n$ to be this weak limit.
	\end{theorem}
	The proof is based on Cegrell's ideas in \cite{cegrell04}. For the convenience of the reader, we present the proof here.
	\begin{proof}
		We first assume that $\max\limits_{1\leq k\leq n}\sup\limits_j\int\limits_\Omega (\triangle u_k^j)^n<\infty$. Let $h\in\E_0(\Omega)$, we observe that
		$\int\limits_\Omega h\triangle u_1^j\wedge...\wedge \triangle u_n^j$ is decreasing by integration by parts (see \cite[Proposition 3.1]{W20}). We also have
		$$\int\limits_\Omega h\triangle u_1^j\wedge...\wedge \triangle u_n^j\geq (\inf_\Omega h)\max\limits_{1\leq k\leq n}\sup\limits_j\int\limits_\Omega (\triangle u_k^j)^n>-\infty,$$
		by \cite[Corollary 3.1]{W20}. Hence, $\lim\limits_{j\rightarrow\infty}\int\limits_\Omega h\triangle u_1^j\wedge...\wedge \triangle u_n^j$ exists for every $h\in\E_0(\Omega)$. Therefore, by \cite[Lemma 3.2]{W20}, the sequence of measures $\triangle u_1^j\wedge...\wedge\triangle u_n^j$ converges weakly.
		
		Suppose that $(v_1^j),...,(v_n^j)\subset\E_0(\Omega)$ such that $v_k^j\searrow u_k$ for every $k=1,...,n$. By the previous argument, we have $\lim\limits_{j\rightarrow\infty}\int\limits_\Omega h\triangle v_1^j\wedge...\wedge \triangle v_n^j$ exists for every $h\in\E_0(\Omega)$. By integration by parts, we observe that, for every $h\in\E_0(\Omega)$,
		\begin{align*}
			&\int\limits_\Omega h\triangle v_1^j\wedge...\wedge \triangle v_n^j=\int\limits_\Omega v_1^j\triangle h\wedge\triangle v_2^j\wedge...\wedge \triangle v_n^j\\
			\geq &\int\limits_\Omega u_1\triangle h\wedge\triangle v_2^j\wedge...\wedge \triangle v_n^j
			=\lim\limits_{s_1\rightarrow\infty}\int\limits_\Omega u_1^{s_1}\triangle h\wedge\triangle v_2^j\wedge...\wedge \triangle v_n^j\\
			=&\lim\limits_{s_1\rightarrow\infty}\int\limits_\Omega v_2^j\triangle h\wedge\triangle u_1^{s_1}\wedge...\wedge \triangle v_n^j\geq...\\
			\geq&\lim\limits_{s_1\rightarrow\infty}\lim\limits_{s_2\rightarrow\infty}...\lim\limits_{s_n\rightarrow\infty}\int\limits_\Omega h\triangle u_1^{s_1}\wedge\triangle u_2^{s_2}\wedge...\wedge \triangle u_n^{s_n}\\
			=&\lim\limits_{s\rightarrow\infty}\int\limits_\Omega h\triangle u_1^{s}\wedge\triangle u_2^{s}\wedge...\wedge \triangle u_n^{s}.
		\end{align*}
		Therefore, $$\lim\limits_{j\rightarrow\infty}\int\limits_\Omega h\triangle v_1^j\wedge...\wedge \triangle v_n^j\geq\lim\limits_{j\rightarrow\infty}\int\limits_\Omega h\triangle u_1^j\wedge...\wedge \triangle u_n^j,$$ for every $h\in\E_0(\Omega)$.
		But this is a symmetric situation, we conclude that
		$$\lim\limits_{j\rightarrow\infty}\int\limits_\Omega h\triangle v_1^j\wedge...\wedge \triangle v_n^j=\lim\limits_{j\rightarrow\infty}\int\limits_\Omega h\triangle u_1^j\wedge...\wedge \triangle u_n^j,$$ for every $h\in\E_0(\Omega)$.
		
		Now we remove the restriction $$\max\limits_{1\leq k\leq n}\sup\limits_j\int\limits_\Omega (\triangle u_k^j)^n<\infty.$$ Let $K$ be a compact subset of $\Omega$. We cover $K$ by $U_1$,...,$U_N$. By the definition of $\E(\Omega)$, we choose $(h_{kl}^j)\subset\E_0(\Omega)$ such that $h_{kl}^j\searrow u_k$ on $U_l$ for $k=1,...,n$ and $l=1,...,N$. We set $w_k^j=\sum\limits_{l=1}^N h_{kl}^j.$ Then $w_k^j\in\E_0(\Omega)$ and $\sup\limits_j\int\limits_\Omega (\triangle w_k^j)^n<\infty$. By rearranging $h_{kl}^j$, we can assume that $w_k^j\leq u_k^j$ on $\cup_{l=1}^N U_l$.
		Now we set $f_k^j=\max(w_k^j, u_k^j)$. It follows that $f_k^j\in\E_0(\Omega)$, $f_k^j=u_k^j$ near $K$ and $\sup\limits_j\int\limits_\Omega(\triangle f_k^j)^n<\infty.$ This completes the proof.
	\end{proof}
	\section{The finite energy classes}\label{sec3}
	\subsection{H\"older type inequalities for $p$-energy estimates of the class $\E_0$}
	First, we recall the definitions of the quaternionic $p$-energy and the mutual quaternionic $p$-energy.
	\begin{definition}\cite{W20}
		Let $u,u_0,u_1,...,u_n\in\E_0(\Omega)$ and $p>0$. We define the quaternionic $p$-energy of $u$ and the mutual quaternionic $p$-energy of $u_0,...,u_n$ to be, respectively,
		$$e_p(u):=\int\limits_\Omega(-u)^p(\triangle u)^n,$$
		$$e_p(u_0,...,u_n):=\int\limits_\Omega(-u_0)^p\triangle u_1\wedge\cdots\wedge \triangle u_n .$$
		and the mutual quaternionic $p$-energy.
	\end{definition}
	The following result provides us an inequality between the mutual quaternionic $p$-energies when $0<p<1$.
	\begin{theorem}\label{energy estimate}
		Let $u_0,...,u_n\in\E_0(\Omega)$. Then, for $0<p<1$, we have
		$$e_p(u_0,u_1,...,u_n)\leq p^{\frac{-1}{1-p}}e_p(u_0,u_0,u_2,...,u_n)^{\frac{p}{p+1}}e_p(u_1,u_1,u_2,...,u_n)^{\frac{1}{p+1}}.$$
	\end{theorem}
	\begin{proof}
		We set $T=\triangle u_2\wedge\cdots\wedge\triangle u_n$ and $w=-(-u_1)^p$. It follows that $T$ is a $(2n-2)$-positive closed current, $w\in\QP(\Omega)\cap L^\infty(\Omega)$ and $\lim\limits_{\Omega\ni q\rightarrow\xi}w(q)=0$ for every $\xi\in\partial\Omega$. We observe that
		\begin{align*}
			&\int\limits_\Omega (-u_0)^p\triangle u_1\wedge T=\frac{1}{2}\int\limits_\Omega (-u_0)^p(d_0d_1-d_1d_0) (-w)^\frac{1}{p}\wedge T\\
			&=\frac{1}{2}\Big[\frac{1}{p}\int\limits_\Omega (-u_0)^p(-w)^{\frac{1}{p}-1}d_0d_1 w\wedge T -\frac{1}{p}\int\limits_\Omega (-u_0)^p(-w)^{\frac{1}{p}-1}d_1d_0 w\wedge T\Big]\\
			&-\frac{1}{2}\Big[ \frac{1-p}{p^2}\int\limits_\Omega (-u_0)^p(-w)^{\frac{1}{p}-2}d_0w\wedge d_1w\wedge T-\frac{1-p}{p^2}\int\limits_\Omega (-u_0)^p(-w)^{\frac{1}{p}-2}d_1w\wedge d_0w\wedge T\Big]\\
			&=\frac{1}{p}\int\limits_\Omega (-u_0)^p(-w)^{\frac{1}{p}-1}\triangle w\wedge T
			- \frac{1-p}{p^2}\int\limits_\Omega (-u_0)^p(-w)^{\frac{1}{p}-2}\gamma(w,w)\wedge T,
		\end{align*}
		where $\gamma(w,w):=\frac{1}{2}\left(d_0w\wedge d_1w-d_1w\wedge d_0w\right)$.
		Then, since $\gamma(w,w)\wedge \triangle u_2\wedge\cdots\wedge\triangle u_n$ is a positive $(2n)$-current (see \cite[Proposition 3.1]{W17}), we obtain
		\begin{equation*}
			\int\limits_\Omega (-u_0)^p\triangle u_1\wedge T\leq \frac{1}{p}\int\limits_\Omega (-u_0)^p(-u_1)^{1-p}\triangle w\wedge T.
		\end{equation*}
		Therefore, by H\"older inequality and integration by parts, it follows that
		\begin{align*}
			\int\limits_\Omega (-u_0)^p\triangle u_1\wedge T&\leq \Big[ \int\limits_\Omega (-u_0)\triangle w\wedge T\Big]^p\Big[\int\limits_\Omega (-u_1)\triangle w\wedge T \Big]^{1-p}\\
			&=\Big[ \int\limits_\Omega (-w)\triangle u_0\wedge T\Big]^p\Big[\int\limits_\Omega (-w)\triangle u_1\wedge T \Big]^{1-p},
		\end{align*}
		and hence
		\begin{equation}\label{ineq1}
			\int\limits_\Omega (-u_0)^p\triangle u_1\wedge T\leq \Big[ \int\limits_\Omega (-u_1)^p\triangle u_0\wedge T\Big]^p\Big[\int\limits_\Omega (-u_1)^p\triangle u_1\wedge T \Big]^{1-p}.
		\end{equation}
		Similarly, we have
		\begin{equation}\label{ineq2}
			\int\limits_\Omega (-u_1)^p\triangle u_0\wedge T\leq \Big[ \int\limits_\Omega (-u_0)^p\triangle u_1\wedge T\Big]^p\Big[\int\limits_\Omega (-u_0)^p\triangle u_0\wedge T \Big]^{1-p}.
		\end{equation}
		Combining inequalities \ref{ineq1} and \ref{ineq2}, we obtain
		\begin{align*}
			\int\limits_\Omega (-u_0)^p\triangle u_1\wedge T \leq \frac{1}{p^{1+p}}&\Big[\int\limits_\Omega (-u_0)^p\triangle u_1\wedge T \Big]^{p^2}\Big[ \int\limits_\Omega (-u_0)^p\triangle u_0\wedge T\Big]^{p(1-p)}\\ 
			&\Big[ \int\limits_\Omega (-u_1)^p\triangle u_1\wedge T\Big]^{1-p}.
		\end{align*}
		Therefore,
		$$ \int\limits_\Omega (-u_0)^p\triangle u_1\wedge T \leq p^{-\frac{1}{1-p}}\Big[  \int\limits_\Omega (-u_0)^p\triangle u_0\wedge T \leq\Big]^{\frac{p}{p+1}}\Big[  \int\limits_\Omega (-u_1)^p\triangle u_1\wedge T \leq\Big]^{\frac{1}{p+1}},$$
		as desired.
	\end{proof}
	By Theorem~\ref{energy estimate} and \cite[Theorem 4.1]{P99}, we have the following H\"older type inequality for $p$-energy estimates of the class $\E_0$ when $0<p<1$.
	\begin{theorem}\label{energy estimate e0 minor}
		Let $u_0,...,u_n\in\E_0(\Omega)$ and $0<p<1$. Then
		$$e_p(u_0,...,u_n)\leq D_p e_p(u_0)^{\frac{p}{n+p}}e_p(u_1)^{\frac{1}{n+p}}...e_p(u_n)^{\frac{1}{n+p}},$$
		where $$D_p=
		p^{\frac{-\alpha(p,n)}{1-p}} \text{ and }
		\alpha(p,n)=(p+2)\Big( \frac{p+1}{p}\Big)^{n-1}-(p+1).$$
	\end{theorem}
	\begin{proof}
		We define \begin{align*}
			F:(\E_0(\Omega))^{n+1}&\rightarrow\R^+\\
			(u_0,...,u_n)&\mapsto e_p(u_0,...,u_n).
		\end{align*} 
		It follows that $F$ is commutative in the last $n$ variables. By Theorem~\ref{energy estimate}, we also have
		$$F(u_0,u_1,...,u_n)\leq p^{\frac{-1}{1-p}}F(u_0,u_0,u_2,...,u_n)^{\frac{p}{p+1}}F(u_1,u_1,u_2,...,u_n)^{\frac{1}{p+1}}.$$
		Therefore, by \cite[Theorem 4.1]{P99}, we obtain
		$$F(u_0,...,u_n)\leq D_p F(u_0,...,u_0)^{\frac{p}{n+p}}F(u_1,...,u_1)^{\frac{1}{n+p}}...F(u_n,...,u_n)^{\frac{1}{n+p}},$$
		and hence
		$$e_p(u_0,...,u_n)\leq D_p e_p(u_0)^{\frac{p}{n+p}}e_p(u_1)^{\frac{1}{n+p}}...e_p(u_n)^{\frac{1}{n+p}},$$
		as desired.
	\end{proof}
	Combine Theorem~\ref{energy estimate e0 minor} and \cite[Theorem 3.4]{W20}, we have the following H\"older type inequality for $p$-energy estimates of the class $\E_0$ when $p>0$.
	\begin{theorem}\label{energy estimate e0}
		Let $u_0,...,u_n\in\E_o(\Omega)$. Then
		$$e_p(u_0,...,u_n)\leq D_p e_p(u_0)^{\frac{p}{n+p}}e_p(u_1)^{\frac{1}{n+p}}...e_p(u_n)^{\frac{1}{n+p}},$$
		where $$D_p=\begin{cases}
			p^{\frac{-\alpha(p,n)}{1-p}},& \text{ if }0<p<1,\\
			1,& \text{ if }p=1,\\
			p^{\frac{p\alpha(p,n)}{p-1}},& \text{ if }p>1,
		\end{cases}$$
		and $$\alpha(p,n)=(p+2)\Big( \frac{p+1}{p}\Big)^{n-1}-(p+1).$$
	\end{theorem}
	The following result is a direct consequence of Theorem~\ref{energy estimate e0}.
	\begin{theorem}\label{7.10}
		Let $p>0$. Then there exists a constant $C>0$ depending only on $p$ and $n$ such that
		$$e_p(u_0,...,u_n)\leq C \max\limits_{0\leq j\leq n}e_p(u_j),$$
		for every $u_0,...,u_n\in\E_0(\Omega).$
	\end{theorem}
	We end this subsection by an estimate of Monge-Amp\`ere of a funcion of the class $\E_0(\Omega)$ on an open subset.
	\begin{lemma}\label{7.20 21}
		Let $u\in\E_0(\Omega)$, $U$ be an open subset of $\Omega$ and $p>0$. Then
		$$\int\limits_U (\triangle u)^n\leq D_p\text{Cap}(U)^{\frac{p}{p+n}}e_p(u)^{\frac{n}{p+n}}.$$
	\end{lemma}
	\begin{proof}
		We can assume that $U$ is relatively compact in $\Omega$. Let $u_{U,\Omega}$ be the extremal function of $U$ in $\Omega$. Then $u_{U,\Omega}\in\E_0(\Omega)$, $u_{U,\Omega}=-1$ on $U$ and $-1\leq u_{U,\Omega}\leq 0$ on $\Omega$. It follows from Theorem~\ref{energy estimate e0} that
		\begin{align*}
			&\int\limits_U (\triangle u)^n\leq \int\limits_\Omega(-u_{U,\Omega})^p (\triangle u)^n
			\leq D_p e_p(u_{U,\Omega})^{\frac{p}{n+p}}e_p(u)^{\frac{n}{n+p}}\\
			\leq &D_p\Big(\int\limits_\Omega(\triangle u_{U,\Omega})^n \Big)^{\frac{p}{n+p}}e_p(u)^{\frac{n}{n+p}}=D_p\text{Cap}(U)^{\frac{p}{p+n}}e_p(u)^{\frac{n}{p+n}}.
		\end{align*}
	\end{proof}
	\subsection{Properties of the finite energy classes}
	\begin{theorem}\label{Ep convex}
		The class $\E_p(\Omega)$ is convex for every $p>0$.
	\end{theorem}
	\begin{proof}
		Let $u,v\in\E_p(\Omega)$, $(u^j),(v^j)\subset\E_0(\Omega)$ such that $u^j\searrow u$, $v^j\searrow v$ and
		$$\sup_j e_p(u^j)<\infty\text{ and }\sup_j e_p(v^j)<\infty.$$
		We need to prove that $$\sup_j e_p(u^j+v^j)<\infty.$$
		Indeed, for every $j$, by the fact that there exists constant $C_1>0$, depending only on $p$ such that $(a+b)^p\leq C_1(a^p+b^p)$ for every $a,b>0$, we have
		\begin{align*}
			&\int\limits_\Omega (-u^j-v^j)^p(\triangle(u^j+v^j))^n=\sum\limits_{m=0}^n\begin{pmatrix}
				m\\n
			\end{pmatrix}\int\limits_\Omega (-u^j-v^j)^p(\triangle u^j)^m\wedge(\triangle v^j)^{n-m}\\
			&\leq C_1 \sum\limits_{m=0}^n\begin{pmatrix}
				m\\n
			\end{pmatrix}\Big(\int\limits_\Omega (-u^j)^p(\triangle u^j)^m\wedge(\triangle v^j)^{n-m}+\int\limits_\Omega (-v^j)^p(\triangle u^j)^m\wedge(\triangle v^j)^{n-m}\Big),
		\end{align*}
		Hence, by Theorem\ref{7.10}, we have, for every $j$,
		$$e_p(u^j+v^j)\leq C_2 \sum\limits_{m=0}^n\begin{pmatrix}
			m\\n
		\end{pmatrix} \max \{e_p(u^j),e_p(v^j)\},$$
		where $C_2$ depends only on $p,n$. Therefore,
		$$\sup_j e_p(u^j+v^j)\leq C_2 \sum\limits_{m=0}^n\begin{pmatrix}
			m\\n
		\end{pmatrix} \max \{\sup_j e_p(u^j), \sup_je_p(v^j)\}<\infty,$$
		as desired.
	\end{proof}
	\begin{theorem}\label{7.12}
		Let $u\in\QP^-(\Omega)$ and $v\in\E_p(\Omega)$, $p>0$. If $u\geq v$ then $u\in\E_p(\Omega)$.
	\end{theorem}
	\begin{proof}
		By the definition of the class $\E_p(\Omega)$ and \cite[Theorem 3.1]{W20}, we can choose $(u^j)\subset \E_0(\Omega)\cap C(\Omega)$ and $(v^j)\subset \E_0(\Omega)$ such that 
		$$u^j\searrow u,\ v^j\searrow v,\ u^j\geq v^j \text{ and }\sup_j e_p(v^j)<\infty.$$
		We will show that $\sup_j e_p(u^j)<\infty.$
		Indeed, by Theorem~\ref{energy estimate e0}, we have, for every $j$,
		$$e_p(u^j)\leq\int\limits_\Omega(-v^j)^p(\triangle u^j)^n\leq C e_p(v^j)^{\frac{p}{n+p}} e_p(u^j)^{\frac{n}{n+p}},$$
		where $C$ depends only on $p,n$. Hence, for every $j$,
		$$e_p(u^j)\leq C^{\frac{n+p}{p}} e_p(v^j),$$
		which implies 
		$$\sup_j e_p(u^j)\leq C^{\frac{n+p}{p}}\sup_j e_p(v^j)<\infty,$$
		as desired.
	\end{proof}
	From the proof of Theorem~\ref{7.12}, we have the following result.
	\begin{lemma}\label{ep e0}
		Let $u,v\in \E_0(\Omega)$, $u\geq v$ and $p>0$. Then $e_p(u)\leq Ce_p(v)$, where $C>0$ is a constant depending only on $p,n$.
	\end{lemma}
	The following theorem shows that the quaternionic Monge–Ampère operator is well-defined on the class $\E_p$ for $p>0.$
	\begin{theorem}\label{converge ep}
		Let $u_1,...,u_n\in\E_p(\Omega)$, $p>0$ and $(u_1^j),...,(u_n^j)\subset\E_0(\Omega)$ such that $u_k^j\searrow u_k$ for every $k=1,...,n$ and
		$$\max\limits_{1\leq k\leq n}\sup\limits_j e_p(u_k^j) <\infty.$$
		Then the sequence of measures $\triangle u_1^j\wedge...\wedge\triangle u_n^j$ converges weakly to a positive Radon measure which does not depend on the choice of the sequences $(u_k^j)$. We then define $\triangle u_1\wedge...\wedge\triangle u_n$ to be this weak limit.
	\end{theorem}
	\begin{proof}
		We first assume that $\max\limits_{1\leq k\leq n}\sup\limits_j\int\limits_\Omega (\triangle u_k^j)^n<\infty$. Let $h\in\E_0(\Omega)$, we observe that
		$\int\limits_\Omega h\triangle u_1^j\wedge...\wedge \triangle u_n^j$ is decreasing by integration by parts (see \cite[Proposition 3.1]{W20}). We also have
		$$\int\limits_\Omega h\triangle u_1^j\wedge...\wedge \triangle u_n^j\geq (\inf_\Omega h)\max\limits_{1\leq k\leq n}\sup\limits_j\int\limits_\Omega (\triangle u_k^j)^n>-\infty,$$
		by \cite[Corollary 3.1]{W20}. Hence, $\lim\limits_{j\rightarrow\infty}\int\limits_\Omega h\triangle u_1^j\wedge...\wedge \triangle u_n^j$ exists for every $h\in\E_0(\Omega)$. Therefore, by \cite[Lemma 3.2]{W20}, the sequence of measures $\triangle u_1^j\wedge...\wedge\triangle u_n^j$ converges weakly.
		
		Suppose that $(v_1^j),...,(v_n^j)\subset\E_0(\Omega)$ such that $v_k^j\searrow u_k$ for every $k=1,...,n$. By the previous argument, we have $\lim\limits_j\int\limits_\Omega h\triangle v_1^j\wedge...\wedge \triangle v_n^j$ exists for every $h\in\E_0(\Omega)$. By integration by parts, we observe that, for every $h\in\E_0(\Omega)$,
		\begin{align*}
			&\int\limits_\Omega h\triangle v_1^j\wedge...\wedge \triangle v_n^j=\int\limits_\Omega v_1^j\triangle h\wedge\triangle v_2^j\wedge...\wedge \triangle v_n^j\\
			\geq &\int\limits_\Omega u_1\triangle h\wedge\triangle v_2^j\wedge...\wedge \triangle v_n^j
			=\lim\limits_{s_1\rightarrow\infty}\int\limits_\Omega u_1^{s_1}\triangle h\wedge\triangle v_2^j\wedge...\wedge \triangle v_n^j\\
			=&\lim\limits_{s_1\rightarrow\infty}\int\limits_\Omega v_2^j\triangle h\wedge\triangle u_1^{s_1}\wedge...\wedge \triangle v_n^j\geq...\\
			\geq&\lim\limits_{s_1\rightarrow\infty}\lim\limits_{s_2\rightarrow\infty}...\lim\limits_{s_n\rightarrow\infty}\int\limits_\Omega h\triangle u_1^{s_1}\wedge\triangle u_2^{s_2}\wedge...\wedge \triangle u_n^{s_n}\\
			=&\lim\limits_{s\rightarrow\infty}\int\limits_\Omega h\triangle u_1^{s}\wedge\triangle u_2^{s}\wedge...\wedge \triangle u_n^{s}.
		\end{align*}
		Therefore, $$\lim\limits_{j\rightarrow\infty}\int\limits_\Omega h\triangle v_1^j\wedge...\wedge \triangle v_n^j\geq\lim\limits_{j\rightarrow\infty}\int\limits_\Omega h\triangle u_1^j\wedge...\wedge \triangle u_n^j,$$ for every $h\in\E_0(\Omega)$.
		But this is a symmetric situation, we conclude that
		$$\lim\limits_{j\rightarrow\infty}\int\limits_\Omega h\triangle v_1^j\wedge...\wedge \triangle v_n^j=\lim\limits_{j\rightarrow\infty}\int\limits_\Omega h\triangle u_1^j\wedge...\wedge \triangle u_n^j,$$ for every $h\in\E_0(\Omega)$.
		
		Now we remove the restriction $$\max\limits_{1\leq k\leq n}\sup\limits_j\int\limits_\Omega (\triangle u_k^j)^n<\infty.$$ Let $K$ be a compact subset of $\Omega$. For each $j$ and $k=1,...,n$, we consider
		$$h_k^j:=(\sup\{\varphi\in\QP^-(\Omega):\ \varphi\leq u_k^j\text{ on }K\})^*.$$
		We have $h_k^j\in\E_0(\Omega)$, $\text{supp}(\triangle h_k^j)^n\subset K$ and $h_k^j\geq u_k^j$. Then,
		by \cite[Theorem 1.2]{WZ15},
		$$\int\limits_\Omega(\triangle h_k^j)^n=\int\limits_K(\triangle h_k^j)^n\leq \int\limits_K(\triangle u_k^j)^n\leq \frac{1}{M^p}e_p(u_k^j),$$
		where $M=-\max\limits_{1\leq k\leq n}\sup\limits_K
		u^1_k>0$. Hence, since $\max\limits_{1\leq k\leq n}\sup\limits_j e_p(u_k^j) <\infty$, we have $$\max\limits_{1\leq k\leq n}\sup\limits_j\int\limits_\Omega (\triangle h_k^j)^n<\infty.$$
		By Lemma~\ref{ep e0}, it follows that there exists a constant $C$ depending only on $p,n$ such that
		$e_p(h_k^j)\leq Ce_p(u_k^j),$
		and hence $$\max\limits_{1\leq k\leq n}\sup\limits_j e_p(h_k^j)<\infty.$$
		Therefore, $h_k^j$ decreases to a function $v_k\in\E_p(\Omega).$ We also have $v_k=u_k$ on $K$.
		This completes the proof.
	\end{proof}
	\begin{remark}\label{energy ep finite}
		If $u,v\in\E_p(\Omega)$ then $$\int\limits_\Omega (-u)^p(\triangle v)^n<\infty.$$  In particular, $e_p(w)<\infty$ for every $w\in\E_p(\Omega).$
		
		Indeed, let $(u^j),(v^j)\in\E_0(\Omega)$ such that $u^j\searrow u$, $v^j\searrow v$ and $$\sup\limits_je_p(u^j)<\infty,\ \sup\limits_je_p(v^j)<\infty.$$ Then, by Theorem~\ref{energy estimate e0},
		\begin{align*}
			\int\limits_\Omega (-u)^p(\triangle v)^n\leq \liminf\limits_{j\rightarrow\infty}\int\limits_\Omega (-u^j)^p(\triangle v^j)^n
			\leq Ce_p(u^j)^{\frac{p}{p+n}}e_p(v^j)^{\frac{n}{p+n}},
		\end{align*}
		where $C$ depends only on $p,n$. Therefore, $\int\limits_\Omega (-u)^p(\triangle v)^n<\infty$, as desired.
	\end{remark}
	The following result is a maximum principle for the class $\E_p$.
	\begin{theorem}\label{7.22 3.19}
		Let $u_1,...,u_n\in\E_p(\Omega)$, $p>0$ and $v\in\QP^-(\Omega)$. Then
		$$\mathbb{1}_A\triangle u_1\wedge...\wedge\triangle u_n=\mathbb{1}_A\triangle \max(u_1,v)\wedge...\wedge\triangle \max(u_n,v),$$
		where $A=\cap_{k=1}^n\{u_k>v\}$.
	\end{theorem}
	\begin{proof}
		Let  $(u_1^j),...,(u_n^j)\subset\E_0(\Omega)$ such that $u_k^j\searrow u_k$ for every $k=1,...,n$ and
		$$\max\limits_{1\leq k\leq n}\sup\limits_j e_p(u_k^j) <\infty.$$
		We can assume that $u_k^j$ is continuous for every $j$ and $k=1,...,n$. We set $v_k^j=\max (u_k^j,v)$ and $A_j=\cap_{k=1}^n\{u_k^j>v\}$. Since $v_k^j\in\QP^-(\Omega)\cap L^\infty(\Omega)$ and $A_j$ is open, we have
		\begin{equation}\label{Aj}
			\mathbb{1}_{A_j}\triangle u_1^j\wedge...\wedge\triangle u_n^j= \mathbb{1}_{A_j}\triangle v_1^j\wedge...\wedge\triangle v_n^j.
		\end{equation}
		We set $u^j=\min(u_1^j,...,u_n^j)$, $u=\min(u_1,...,u_n)$, $\psi^j=\max (u^j-v,0)$ and $\psi=\max(u-v,0)$. Fix $\delta>0$, we set $g^j=\frac{\psi^j}{\psi^j+\delta}$ and $g=\frac{\psi}{\psi+\delta}$. We multiply both sides of equality (\ref{Aj}) with $g^j$ to obtain
		\begin{equation}\label{gj}
			g^j\triangle u_1^j\wedge...\wedge\triangle u_n^j=g^j\triangle v_1^j\wedge...\wedge\triangle v_n^j.
		\end{equation}
		Let $\chi\in C_0^\infty(\Omega)$ and fix $\epsilon>0$. Since $v$ is quasi-continuous, $\psi^j,\ \psi$ are also quasi-continuous. Then, by \cite[Theorem 1.1]{WZ15}, there exists an open subset $U$ of $\Omega$ and there exist functions $\varphi^j,\ \varphi\in C(\Omega)$ such that, $\varphi^j|_{\Omega\setminus U}=\psi^j|_{\Omega\setminus U}$ and $\varphi|_{\Omega\setminus U}=\psi|_{\Omega\setminus U}$. Since $\psi^j\searrow\psi$, $\varphi^j$ converges uniformly to $\varphi$ on $(\Omega\setminus U)\cap\text{supp}\chi$, and hence $h^j=\frac{\varphi^j}{\varphi^j+\delta}$ converges uniformly to $h=\frac{\varphi}{\varphi+\delta}$ on $(\Omega\setminus U)\cap\text{supp}\chi$.
		
		Since $g^j,\ h^j$ are uniformly bounded, we have
		$$ \Big|\int\limits_\Omega \chi g^j \triangle u_1^j\wedge...\wedge\triangle u_n^j-\int\limits_\Omega \chi h^j \triangle u_1^j\wedge...\wedge\triangle u_n^j\Big|\leq C_1\int\limits_U \triangle u_1^j\wedge...\wedge\triangle u_n^j.$$
		Then, since $u_1^j+...+ u_n^j\in \E_0(\Omega)$ (by \cite[Corollary 3.3]{W20}), it follows from Lemma~\ref{7.20 21} that
		\begin{equation}  \label{1} 
			\Big|\int\limits_\Omega \chi g^j \triangle u_1^j\wedge...\wedge\triangle u_n^j-\int\limits_\Omega \chi h^j \triangle u_1^j\wedge...\wedge\triangle u_n^j\Big|\leq C_1\int\limits_U (\triangle (u_1^j+...+ u_n^j))^n\leq C_2\epsilon^{\frac{p}{p+n}}.
		\end{equation}
		Similarly, since $u_1+...+ u_n\in \E_p(\Omega)$ (by Theorem~\ref{Ep convex}), we have
		\begin{align*}
			\Big|\int\limits_\Omega \chi g \triangle u_1\wedge...\wedge\triangle u_n-\int\limits_\Omega \chi h \triangle u_1\wedge...\wedge\triangle u_n\Big|&\leq C_3\int\limits_U (\triangle (u_1+...+ u_n))^n\\
			&\leq C_3\liminf\limits_{j\rightarrow\infty}\int\limits_U (\triangle (u_1^j+...+ u_n^j))^n.
		\end{align*}
		Again, by Lemma~\ref{7.20 21}, we obtain
		\begin{equation}\label{2}
			\Big|\int\limits_\Omega \chi g \triangle u_1\wedge...\wedge\triangle u_n-\int\limits_\Omega \chi h \triangle u_1\wedge...\wedge\triangle u_n\Big|\leq C_4 \epsilon^{\frac{p}{p+n}}.
		\end{equation}
		Moreover, since $\triangle u_1^j\wedge...\wedge u_n^j$ converges weakly to $\triangle u_1\wedge...\wedge \triangle u_n$ (by Theorem~\ref{converge ep}), it follows that
		$$\lim\limits_{j\rightarrow\infty}\int\limits_\Omega\chi h\triangle u_1^j\wedge...\wedge \triangle u_n^j=\int\limits_\Omega\chi h\triangle u_1\wedge...\wedge\triangle u_n.$$
		Hence, we have
		\begin{align*}
			&\limsup\limits_{j\rightarrow\infty}\Big| \int\limits_\Omega \chi h^j\triangle u_1^j\wedge...\wedge \triangle u_n^j-\int\limits_\Omega \chi h\triangle u_1\wedge...\wedge \triangle u_n\Big|\leq \limsup\limits_{j\rightarrow\infty} \int\limits_\Omega\chi|h^j-h|\triangle u_1^j\wedge...\wedge \triangle u_n^j\\
			&=\limsup\limits_{j\rightarrow\infty} \Big(\int\limits_U\chi|h^j-h|\triangle u_1^j\wedge...\wedge \triangle u_n^j+\int\limits_{\Omega\setminus U}\chi|h^j-h|\triangle u_1^j\wedge...\wedge \triangle u_n^j\Big).
		\end{align*}
		Since $h^j$, $h$ are uniformly bounded and $h^j$ converges uniformly to $h$ on $(\Omega\setminus U)\cap\text{supp}\chi$, we see that
		\begin{align*}
			&\limsup\limits_{j\rightarrow\infty}\Big| \int\limits_\Omega \chi h^j\triangle u_1^j\wedge...\wedge \triangle u_n^j-\int\limits_\Omega \chi h\triangle u_1\wedge...\wedge \triangle u_n\Big|\\
			\leq &\limsup\limits_{j\rightarrow\infty} \Big(C_5\int\limits_U\triangle u_1^j\wedge...\wedge \triangle u_n^j+\|h^j-h\|_{L^\infty((\Omega\setminus U)\cap\text{supp}\chi)}\int\limits_\Omega\chi \triangle u_1^j\wedge...\wedge \triangle u_n^j\Big).
		\end{align*}
		Again, by Lemma~\ref{7.20 21}, we obtain
		\begin{equation}\label{3}
			\limsup\limits_{j\rightarrow\infty}\Big| \int\limits_\Omega \chi h^j\triangle u_1^j\wedge...\wedge \triangle u_n^j-\int\limits_\Omega \chi h\triangle u_1\wedge...\wedge \triangle u_n\Big|\leq C_6\epsilon^{\frac{p}{p+n}}.
		\end{equation}
		From inequalities (\ref{1}), (\ref{2}), (\ref{3}), we have
		$$\limsup\limits_{j\rightarrow\infty}\Big| \int\limits_\Omega \chi g^j\triangle u_1^j\wedge...\wedge \triangle u_n^j-\int\limits_\Omega \chi g\triangle u_1\wedge...\wedge \triangle u_n\Big|\leq C_7\epsilon^{\frac{p}{p+n}}.$$
		Here $C_1, C_2,..., C_7$ are positive constants which do not depend on $j,\epsilon.$ Hence, $g^j\triangle u_1^j\wedge...\wedge \triangle u_n^j$ converges weakly to $g\triangle u_1\wedge...\wedge \triangle u_n$. Similarly, we get $g^j\triangle v_1^j\wedge...\wedge \triangle v_n^j$ converges weakly to $g\triangle v_1\wedge...\wedge \triangle v_n$. Therefore, by equality (\ref{gj}), we see that
		$$g\triangle u_1\wedge...\wedge \triangle u_n=g\triangle v_1\wedge...\wedge \triangle v_n.$$
		The result thus follows by letting $\delta$ to $0$.
	\end{proof}
	By the maximum principle and  the H\"older type inequality
	for $p$-energy estimates of the class $\E_p$, we have the following convergence theorem.
	\begin{theorem}\label{7.17}
		Let $u_1,...,u_n\in\E_p(\Omega)$, $p>0$ and $(u_1^j),...,(u_n^j)\subset\E_0(\Omega)$ such that $u_k^j\searrow u_k$ for every $k=1,...,n$ and
		$$\max\limits_{1\leq k\leq n}\sup\limits_j e_p(u_k^j) <\infty.$$
		Then, for every $h\in\E_0(\Omega)$, we have
		$$\lim\limits_{j\rightarrow\infty}\int\limits_\Omega h\triangle u_1^j\wedge...\wedge\triangle u_n^j=\int\limits_\Omega h\triangle u_1\wedge...\wedge\triangle u_n.$$
	\end{theorem}
	\begin{proof}
		Without loss of generality, we can assume that the right-hand side is finite and $-1\leq h\leq 0$. We set $r=\min(1,\frac{1}{p})$, $\psi=-(-h)^r$, $u_{k,l}^j=\max(u_k^j,l\psi),$ $u_{k,l}=\max(u_k,l\psi)$ for each $k=1,...,n$ and $l,j\in\mathbb{N}$. Then $u_{k,l}^j, u_{k,l}\in\QP^-(\Omega)\cap L^\infty(\Omega)$ and $\lim\limits_{\Omega\ni q\rightarrow\xi}u_{k,l}^j(q)=\lim\limits_{\Omega\ni q\rightarrow\xi}u_{k,l}(q)=0$ for every $\xi\in\partial\Omega$.
		
		Since $\triangle u_{1,l}^j\wedge...\wedge\triangle u_{n,l}^j$ converges weakly to $\triangle u_{1,l}\wedge...\wedge\triangle u_{n,l}$ and $(-h)$ is lower semicontinuous, we have, for every $l$,
		\begin{equation}\label{>=}
			\lim\limits_{j\rightarrow\infty}\int\limits_\Omega(-h)\triangle u_{1,l}^j\wedge...\wedge\triangle u_{n,l}^j\geq\int\limits_\Omega(-h)\triangle u_{1,l}\wedge...\wedge\triangle u_{n,l}.
		\end{equation}
		Moreover, it follows from integration by parts and the decreasing of $(u_{k,l}^j)$ that, for every $l$,
		\begin{equation}\label{<=}
			\int\limits_\Omega(-h)\triangle u_{1,l}^j\wedge...\wedge\triangle u_{n,l}^j\leq\int\limits_\Omega(-h)\triangle u_{1,l}\wedge...\wedge\triangle u_{n,l}.
		\end{equation}
		By inequalities (\ref{>=}) and (\ref{<=}), it follows that, for every $l$,
		\begin{equation}\label{lim1}
			\lim\limits_{j\rightarrow\infty}\int\limits_\Omega(-h)\triangle u_{1,l}^j\wedge...\wedge\triangle u_{n,l}^j=\int\limits_\Omega(-h)\triangle u_{1,l}\wedge...\wedge\triangle u_{n,l}.
		\end{equation}
		
		By Theorem~\ref{7.22 3.19}, we observe that, for every $l,j$,
		$$\int\limits_{\cap_{k=1}^n\{u_k^j>l\psi\}} (-h)\triangle u_{1,l}^j\wedge...\wedge\triangle u_{n,l}^j=\int\limits_{\cap_{k=1}^n\{u_k^j>l\psi\}} (-h)\triangle u_{1}^j\wedge...\wedge\triangle u_{n}^j.$$
		Hence, for every $l,j$,
		\begin{align*}
			&\Big| \int\limits_{\Omega} (-h)\triangle u_{1,l}^j\wedge...\wedge\triangle u_{n,l}^j-\int\limits_{\Omega} (-h)\triangle u_{1}^j\wedge...\wedge\triangle u_{n}^j\Big|\\
			&=\Big|\int\limits_{\cup_{k=1}^n\{u_k^j\leq l\psi\}} (-h)\triangle u_{1,l}^j\wedge...\wedge\triangle u_{n,l}^j-\int\limits_{\cup_{k=1}^n\{u_k^j\leq l\psi\}} (-h)\triangle u_{1}^j\wedge...\wedge\triangle u_{n}^j\Big|\\
			&\leq \int\limits_{\cup_{k=1}^n\{u_k^j\leq l\psi\}} (-h)\triangle u_{1,l}^j\wedge...\wedge\triangle u_{n,l}^j+\int\limits_{\cup_{k=1}^n\{u_k^j\leq l\psi\}} (-h)\triangle u_{1}^j\wedge...\wedge\triangle u_{n}^j\\
			&\leq \frac{1}{l^p}\sum_{k=1}^n\Big( \int\limits_{\{u_k^j\leq l\psi\}} (-h)\triangle u_{1,l}^j\wedge...\wedge\triangle u_{n,l}^j+\int\limits_{\{u_k^j\leq l\psi\}} (-h)\triangle u_{1}^j\wedge...\wedge\triangle u_{n}^j\Big)\\
			&\leq \frac{1}{l^p}\sum_{k=1}^n\Big(\int\limits_\Omega (-u_{k}^j)^p \triangle u_{1,l}^j\wedge...\wedge\triangle u_{n,l}^j+\int\limits_\Omega (-u_{k}^j)^p \triangle u_{1}^j\wedge...\wedge\triangle u_{n}^j\Big).
		\end{align*}
		Then, by Theorem~\ref{7.10} and Lemma~\ref{ep e0}, we have, for every $l,j$,
		\begin{align*}
			&\Big| \int\limits_{\Omega} (-h)\triangle u_{1,l}^j\wedge...\wedge\triangle u_{n,l}^j-\int\limits_{\Omega} (-h)\triangle u_{1}^j\wedge...\wedge\triangle u_{n}^j\Big|\\
			&\leq \frac{C_1}{l^p}\sum_{k=1}^n\Big( \max(e_p(u_k^j),e_p(u_{1,l}^j),...,e_p(u_{n,l}^j))+\max(e_p(u_k^j),e_p(u_{1}^j),...,e_p(u_{n}^j))\Big)\\
			&\leq \frac{C_2}{l^p}\sum_{k=1}^n \max(e_p(u_k^j),e_p(u_{1}^j),...,e_p(u_{n}^j))\\
			&\leq \frac{nC_2}{l^p} \max\limits_{1\leq k\leq n}\sup\limits_j e_p(u_k^j),
		\end{align*}
		where $C_1,C_2$ depend only on $p,n$. Then, by the assumption $\max\limits_{1\leq k\leq n}\sup\limits_j e_p(u_k^j)<\infty$, we see that, for every $j$,
		\begin{equation}\label{lim2}
			\lim\limits_{l\rightarrow\infty}\int\limits_{\Omega} (-h)\triangle u_{1,l}^j\wedge...\wedge\triangle u_{n,l}^j=\int\limits_{\Omega} (-h)\triangle u_{1}^j\wedge...\wedge\triangle u_{n}^j
		\end{equation}
		
		Now, by equalities (\ref{lim1}) and (\ref{lim2}), it follows that
		\begin{align*}
			\lim\limits_{j\rightarrow\infty}\int\limits_\Omega (-h)\triangle u_{1}^j\wedge...\wedge\triangle u_{n}^j&=
			\lim\limits_{j\rightarrow\infty} \lim\limits_{l\rightarrow\infty}\int\limits_{\Omega} (-h)\triangle u_{1,l}^j\wedge...\wedge\triangle u_{n,l}^j\\
			&=\lim\limits_{l\rightarrow\infty}\int\limits_\Omega(-h)\triangle u_{1,l}\wedge...\wedge\triangle u_{n,l}\\
			&=\int\limits_\Omega(-h)\triangle u_{1}\wedge...\wedge\triangle u_{n},
		\end{align*}
		where the last equality follows from Theorem~\ref{7.22 3.19}.
		The proof is completed.
	\end{proof}
	The next result is a consequence of Theorem~\ref{7.17} and Theorem~\ref{energy estimate e0}.
	\begin{theorem}\label{h triangle u finite}
		Let $u_1,...,u_n\in\E_p(\Omega)$, $0<p\leq 1$.
		Then, for every $h\in\E_0(\Omega)$, we have
		$$\int\limits_\Omega (-h)\triangle u_1\wedge...\wedge\triangle u_n<\infty.$$
	\end{theorem}
	\begin{proof}
		Let $(u_1^j),...,(u_n^j)\subset\E_0(\Omega)$ such that $u_k^j\searrow u_k$ for every $k=1,...,n$ and
		$$M=\max\limits_{1\leq k\leq n}\sup\limits_j e_p(u_k^j) <\infty.$$
		By Theorem~\ref{7.17} and Theorem~\ref{energy estimate e0}, we have
		\begin{align*}
			\int\limits_\Omega (-h)\triangle u_1\wedge...\wedge\triangle u_n&=\lim\limits_{j\rightarrow\infty}\int\limits_\Omega (-h)\triangle u_{1}^j\wedge...\wedge\triangle u_{n}^j\\
			&\leq (-\inf\limits_\Omega h)^{1-p} \lim\limits_{j\rightarrow\infty}\int\limits_\Omega (-h)^p\triangle u_{1}^j\wedge...\wedge\triangle u_{n}^j\\
			&\leq C (-\inf\limits_\Omega h)^{1-p} e_p(h)^\frac{p}{p+n}\lim\limits_{j\rightarrow\infty}e_p(u_{1}^j)^\frac{1}{p+n}...e_p(u_{n}^j)^\frac{1}{p+n}\\
			&\leq C (-\inf\limits_\Omega h)^{1-p} e_p(h)^\frac{p}{p+n} M^\frac{n}{p+n},
		\end{align*}
		where $C>0$ depends only on $p,n$. Therefore 
		$$\int\limits_\Omega h\triangle u_1\wedge...\wedge\triangle u_n<\infty,$$
		as desired.
	\end{proof}
	The following three theorems are comparison principles for the class $\E_p$.
	\begin{theorem}\label{compa when u<v}
		Let $u,v\in\E_p(\Omega)$, $p>0$, such that $u\leq v$ on $\Omega$. Then
		$$\int\limits_\Omega (\triangle u)^n\geq \int\limits_\Omega (\triangle v)^n.$$
	\end{theorem}
	\begin{proof}
		Let $(u^j),(v^j)$ be two sequences in $ \E_0(\Omega)$ decreasing to $u,v$ and 
		$$\sup\limits_j e_p(u^j)<\infty, \ \sup\limits_j e_p(v^j)<\infty.$$ We can suppose that $u^j\leq v^j$ for every $j$. Let $h\in\E_0(\Omega)\cap C(\Omega).$ It follows from integration by parts that, for every $j$,
		$$\int\limits_\Omega (-h)(\triangle v^j)^n\leq \int\limits_\Omega (-h)(\triangle u^j)^n.$$
		By Theorem~\ref{7.17}, we obtain
		$$\lim\limits_{j\rightarrow\infty}\int\limits_\Omega (-h)(\triangle u^j)^n=\int\limits_\Omega (-h)(\triangle u)^n,$$
		and
		$$\lim\limits_{j\rightarrow\infty}\int\limits_\Omega (-h)(\triangle v^j)^n=\int\limits_\Omega (-h)(\triangle v)^n.$$
		Therefore,
		$$\int\limits_\Omega (-h)(\triangle v)^n\leq \int\limits_\Omega (-h)(\triangle u)^n.$$
		The proof is thus complete by letting $h$ decrease to $-1$.
	\end{proof}
	\begin{theorem}\label{compare on u<v}
		Let $u,v\in \E_p(\Omega)$ and $p>0$. Then
		$$\int\limits_{\{u<v\}}(\triangle u)^n\geq \int\limits_{\{u<v\}}(\triangle v)^n.$$
	\end{theorem}
	\begin{proof}
		The case $p>1$ was handled in \cite[Proposition 4.2]{W20}. We consider the case $0<p\leq 1$. Let $h\in\E_0(\Omega)\cap C(\Omega)$. 
		We see that for almost everywhere $r$,
		$$\int\limits_{\{v=ru\}}(-h)(\triangle u)^n=0.$$
		Therefore, we can restrict to the case
		\begin{equation}\label{a}
			\int\limits_{\{v=u\}}(-h)(\triangle u)^n=0.
		\end{equation}
		From Theorem~\ref{7.22 3.19}, we obtain
		\begin{equation}\label{b}
			\mathbb{1}_{\{u<v\}}(\triangle v)^n=\mathbb{1}_{\{u<v\}}(\triangle \max(u,v))^n\text{ and }\mathbb{1}_{\{u>v\}}(\triangle u)^n=\mathbb{1}_{\{u<v\}}(\triangle \max(u,v))^n.
		\end{equation}
		As in proof of Theorem~\ref{compa when u<v}, we have
		\begin{equation}\label{c}
			\int\limits_\Omega (-h)(\triangle \max(u,v))^n\leq\int\limits_\Omega (-h)(\triangle u)^n.
		\end{equation}
		Therefore, from inequalities (\ref{a}), (\ref{b}) and (\ref{c}), we obtain
		\begin{align*}
			&\int\limits_{\{u<v\}}(-h)(\triangle v)^n
			=\int\limits_{\{u<v\}}(-h)(\triangle \max(u,v))^n\\
			&\leq \int\limits_\Omega(-h)(\triangle \max(u,v))^n-\int\limits_{\{u>v\}}(-h)(\triangle \max(u,v))^n\\
			&\leq  \int\limits_\Omega(-h)(\triangle u)^n-\int\limits_{\{u>v\}}(-h)(\triangle u)^n\\
			&=\int\limits_{\{u<v\}}(-h)(\triangle u)^n.
		\end{align*}
		We notice that all terms are finite by Theorem~\ref{h triangle u finite}, hence all inequalites make sense.
		The proof is thus complete by letting $h$ decrease to $-1$.
	\end{proof}
	\begin{theorem}\label{compare implies u<v}
		Let $u,v\in\E_p(\Omega)$ ($p>0$) satisfying $(\triangle u)^n\geq (\triangle v)^n$. Then $u\leq v$ on $\Omega$.
	\end{theorem}
	\begin{proof}
		Suppose by contradiction that there exists $q_0\in\Omega$ such that $v(q_0)<u(q_0)$. Let $\rho$ be an exhaustion function of $\Omega$ and $R>0$ such that $|q-q_0|\leq R$ for every $q\in\Omega$. We fix $\epsilon>0$ small enough such that $\rho(z_0)<-\epsilon R^2.$ We set
		$$P(q):=\max\Big(\rho(q),\epsilon\frac{\|q-q_0\|^2-R^2}{8}\Big).$$
		Then $P$ is an continuous exhaustion function of $\Omega$ and $(\triangle P)^n\geq \epsilon^n dV$ near $q_0$.
		Since $v(q_0)<u(q_0)$, we can choose $\eta>0$ small enough such that $v(q_0)<u(q_0)+\eta P(q_0)$. Then the Lebesgue measure of $T=\{q\in\Omega:\ v(q)<u(q)+\eta P(q)\}$ is strictly positive, and hence 
		\begin{equation}\label{i}
			\int\limits_T(\triangle P)^n>0.
		\end{equation}
		It follows from Theorem~\ref{compare on u<v} that
		\begin{equation}\label{ii}
			\int\limits_T(\triangle (u+\eta P))^n\leq \int\limits_T(\triangle v)^n.
		\end{equation}
		Furthermore, 
		\begin{equation}\label{iii}
			\int\limits_T(\triangle (u+\eta P))^n\geq \int\limits_T(\triangle u)^n +\eta^n \int\limits_T(\triangle P)^n.
		\end{equation}
		By inequalities (\ref{i}), (\ref{ii}) and (\ref{iii}), we conclude that
		$$\int\limits_T(\triangle u)^n< \int\limits_T(\triangle v)^n,$$
		a contradiction. The proof is complete.
	\end{proof}
	The following result is a convergence theorem for functions of the class $\E_p$ regarding their $p$-energies.
	\begin{theorem}\label{converge p-energies}
		Let $u,v\in\E_p(\Omega)$, $p>0$. Then there exist sequences $(u^j), (v^j)\subset\E_0(\Omega)$ decreasing to $u, v$ respectively such that
		$$\lim\limits_{j\rightarrow\infty}\int\limits_\Omega (-u^j)^p(\triangle v^j)^n=\int\limits_\Omega (-u)^p(\triangle v)^n.$$
		In particular, if $w\in\E_p(\Omega)$, then there exists sequence $(w^j)\in\E_0(\Omega)$ decreasing to $w$ such that
		$$\lim\limits_{j\rightarrow\infty}e_p(w^j)=e_p(w).$$
	\end{theorem}
	\begin{proof}
		Let $(u^j)\subset\E_0(\Omega)$ such that $u^j\searrow u$ and $\sup\limits_je_p(u^j)<\infty.$ Since $(\triangle v)^n$ does not charge on Q-polar sets, it follows from \cite[Proposition 4.3]{W20} that $$(\triangle v)^n=f(\triangle \psi)^n,$$ where $\psi\in\E_0(\Omega)$ and $0\leq f\in L^1_{\text{loc}}((\triangle\psi)^n)$. Then, by \cite[Lemma 4.7]{W20}, we can find a sequence $(v^j)\in\E_0(\Omega)$ such that $$(\triangle v^j)^n=\min(f,j)(\triangle\psi)^n.$$
		By Theorem~\ref{compare implies u<v}, $(v^j)$ is a decreasing sequence and $v^j\geq v$. Let $\varphi=\lim\limits_{j\rightarrow\infty}v^j$. Then $\varphi\geq v$ and hence $\varphi\not\equiv-\infty.$ We observe that, for every $j$,
		$$e_p(v^j)=\int\limits_\Omega (-v^j)^p\min (f,j)(\triangle\psi)^n\leq \int\limits_\Omega(-v)^p f (\triangle\psi)^n=e_p(v).$$
		Then, by Remark~\ref{energy ep finite}, we have $\sup\limits_j e_p(v^j)\leq e_p(v)<\infty,$ and hence $\varphi\in\E_p(\Omega).$ Moreover, $\varphi$ solves $(\triangle\varphi)^n=(\triangle v)^n$. Again, by Theorem~\ref{compare implies u<v}, we obtain $\varphi\equiv v$. Hence, we have
		$$\int\limits_\Omega (-u)^p(\triangle v)^n=\lim\limits_{j\rightarrow \infty}\int\limits_\Omega (-u^j)^p\min(f,j)(\triangle\psi)^n=\lim\limits_{j\rightarrow \infty}\int\limits_\Omega (-u^j)^p(\triangle v^j)^n,$$
		as desired.
	\end{proof}
	\begin{remark}\label{chosen sequence}
		The sequences $(u^j),(v^j)$ chosen in Theorem~\ref{converge p-energies} satisfy $\sup\limits_j e_p(u^j)<\infty$ and $\sup\limits_j e_p(v^j)<\infty.$
	\end{remark}
	In closing this section, we extend Theorem~\ref{energy estimate e0} from the class $\E_0$ to the class $\E_p$ as follows.
	\begin{theorem}\label{thm:energyine}
		Let $u_0,...,u_n\in\E_p(\Omega)$. Then
		$$e_p(u_0,...,u_n)\leq D_p e_p(u_0)^{\frac{p}{n+p}}e_p(u_1)^{\frac{1}{n+p}}...e_p(u_n)^{\frac{1}{n+p}},$$
		where $$D_p=\begin{cases}
			p^{\frac{-\alpha(p,n)}{1-p}},& \text{ if }0<p<1,\\
			1,& \text{ if }p=1,\\
			p^{\frac{p\alpha(p,n)}{p-1}},& \text{ if }p>1,
		\end{cases}$$
		and $$\alpha(p,n)=(p+2)\Big( \frac{p+1}{p}\Big)^{n-1}-(p+1).$$
	\end{theorem}
	\begin{proof}
		By Theorem~\ref{converge p-energies}, we can find $(u_0^j),...,(u_n^j)\subset\E_0(\Omega)$ such that $u_k^j\searrow u_k$ and $e_p(u_k^j)\rightarrow e_p(u_k)$ for all $k=0,...,n$. 
		
		Fix $j_0$, we first claim that
		\begin{equation}\label{liminf}
			\int\limits_\Omega (-u_0^{j_0})^p\triangle u_1\wedge...\wedge u_n\leq \liminf\limits_{j\rightarrow\infty}\int\limits_\Omega (-u_0^{j_0})^p\triangle u^j_1\wedge...\wedge u^j_n.
		\end{equation}
		Indeed, let $\chi\in C_0^\infty(\Omega),$ since $(-u_0^{j_0})^p$ is lower semicontinuous and so is $(-u_0^{j_0})^p\chi$. Then, by Remark~\ref{chosen sequence} and Theorem~\ref{converge ep}, it follows that
		$$\int\limits_\Omega (-u_0^{j_0})^p\chi \triangle u_1\wedge...\wedge \triangle u_n\leq\liminf\limits_{j\rightarrow\infty} \int\limits_\Omega(-u_0^{j_0})^p\chi \triangle u^j_1\wedge...\wedge \triangle u^j_n.$$
		The claim thus follows by letting $\chi\rightarrow1.$
		
		For every $j_0$, by inequality (\ref{liminf}) and Theorem~\ref{energy estimate e0}, we have
		\begin{align*}
			\int\limits_\Omega (-u_0^{j_0})^p\triangle u_1\wedge...\wedge u_n&\leq D_p\liminf\limits_{j\rightarrow\infty}e_p(u_0^{j_0})^\frac{p}{p+n}e_p(u^j_1)^\frac{1}{p+n}...e_p(u^j_n)^\frac{1}{p+n}\\
			&=D_pe_p(u_0^{j_0})^\frac{p}{p+n}e_p(u_1)^\frac{1}{p+n}...e_p(u_n)^\frac{1}{p+n}.
		\end{align*}
		The proof is thus complete by letting $j_0\rightarrow\infty.$
	\end{proof}
	
	\section{On the constant in the quaternionic energy estimate}\label{sec4}
	We first recall some results concerning to beta, gamma and digamma functions.
	For $x,y >0$, the Beta function $B(x,y)$ is defined by
	$$B(x,y)=\int_0^1 t^{x-1}(1-t)^{y-1}dt.$$
	A property of the Beta function relating to the Gamma function as follows:
	$$B(x,y)=\frac{\Gamma (x)\Gamma(y)}{\Gamma (x+y)},$$
	where the Gamma function is given by
	$$\Gamma(x)=\int_0^{+\infty}e^{-t}t^{x-1}dt.$$
	This shows us that the Beta function is symmetric. Its partial derivative is computed as follows
	\begin{align}
		\label{eq:derivative}
		\frac{\partial }{\partial y}B(x,y)= &B(x,y)\left(\frac{\Gamma^{'}(y)}{\Gamma (y)}-\frac{\Gamma^{'}(x+y)}{\Gamma (x+y)}\right)\notag\\
		=&B(x,y)(\psi(y)-\psi(x+y)),
	\end{align}
	where $\psi(y)=\frac{\Gamma^{'}(y)}{\Gamma (y)}$ is the digamma function.
	\begin{lemma}
		\label{lem}
		Let $f \colon \bN \times (0,+\infty)\to \R$ be the function defined by
		\begin{equation}
			f(p,n)=\frac{1}{n}+\frac{p}{n+p}+\psi (n)-\psi (n+p+1).
		\end{equation}
		Then we have $f(p,n)\neq 0$ for all $n\in \bN$ and $p>0$ $(p\neq 1)$.
	\end{lemma}
	{\it Proof.}
	This follows from \cite[Lemma 2.2]{AC09}. \hfill$\square$
	\vspace{3mm}
	
	We now ready to construct an example to show that the optimal constant on the right-hand side in inequality~(\ref{optimal}) must exceed $1$ in the class $\E_p$ when $p\neq1$.
	\begin{example}
		For $a>0$, we consider a family of quaternionic plurisubharmonic functions in the unit ball $\mathbb{B}\subset\mathbb{H}^{n}$:
		\begin{align*}
			u_a(q)=\norm{q}^{2a}-1.
		\end{align*}
		We can compute
		\begin{align*}
			\left(\Delta u_a\right)^n =C_0 a^{n}(a+1)\norm{q}^{2n(a-1)}dV,
		\end{align*}
		where $C_0$ is the constant depending only on $n$ and $dV$ is the Lebesgue measure on $\mathbb{H}^{n}$.
		For $a,b>0$, we have
		\begin{align}\label{energy1}
			&\int_{\mathbb{B}} (-u_a)^p \left(\Delta u_b\right)^n = C_0 b^{n}(b+1) \int_{\mathbb{B}} \left(1-\norm{q}^{2a}\right)^p \norm{q}^{2n(b-1)}dV \notag\\
			&= C_0 b^{n}(b+1) \mbox{Area}(\partial B) \int_0^1 \left(1-t^{2a}\right)^p t^{2n(b-1)} t^{4n-1}dt \notag\\
			&= C \frac{b^{n}(b+1)}{a}\int_0^1 (1-s)^p s^{\frac{(b+1)n}{a}-1}ds \ (\text{by setting } s=t^{2a})\notag\\
			&= C \frac{b^{n}(b+1)}{a} B\left(p+1, \frac{(b+1)n}{a}\right),
		\end{align}
		where $C$ depends only on $n$. Letting $a=b$ in (\ref{energy1}), the quaternionic $p$-energy of $u_a$ becomes
		\begin{align*}
			e_p (u_a) &=\int_{\mathbb{B}}(-u_a)^p \left(\Delta u_a\right)^n\\
			&=C a^{n-1}(a+1) B\left(p+1,\frac{(a+1)n}{a}\right).
		\end{align*}
		
		Assume that the optimal constant on the right-hand side in inequality~(\ref{optimal}) is $1$. Then the quaternionic $p$-energy estimate inequality becomes
		\begin{align*}
			\frac{b^{n}(b+1)}{a}B\left(p+1, \frac{(b+1)n}{a}\right)\leq& \left(a^{n-1}(a+1)\right)^{\frac{p}{n+p}}B\left(p+1, \frac{(a+1)n}{a}\right)^{\frac{p}{n+p}}\\
			& \times \left(b^{n-1}(b+1)\right)^{\frac{n}{n+p}}B\left(p+1, \frac{(b+1)n}{b}\right)^{\frac{n}{n+p}}
		\end{align*}
		which is equivalent to the following inequality
		\begin{align*}
			&\left(\frac{b}{a}\right)^{\frac{np+n}{n+p}}\left(\frac{b+1}{a+1}\right)^{\frac{p}{n+p}}B\left(p+1, \frac{(b+1)n}{a}\right)\\
			\leq& B\left(p+1,\frac{(a+1)n}{a}\right)^{\frac{p}{n+p}}B\left(p+1, \frac{(b+1)n}{b}\right)^{\frac{n}{n+p}}.
		\end{align*}
		We define the function $F: (0,\infty)\times (0,\infty)\to \mathbb{R}$ by
		\begin{align*}
			F(a,b)=&\left(\frac{b}{a}\right)^{\frac{np+n}{n+p}}\left(\frac{b+1}{a+1}\right)^{\frac{p}{n+p}}B\left(p+1, \frac{(b+1)n}{a}\right)\\
			&- B\left(p+1,\frac{(a+1)n}{a}\right)^{\frac{p}{n+p}}B\left(p+1, \frac{(b+1)n}{b}\right)^{\frac{n}{n+p}}.
		\end{align*}
		It is easy to see that  $F(a,b)$ is continuously differentiable on $(0,\infty)\times (0,\infty)$. Moreover, since $F(1,1)=0$, it will be finished provided that $\frac{\partial }{\partial b} F(1,1)\neq 0$. Indeed, by formula (\ref{eq:derivative}), we can compute
		\begin{align*}
			&\frac{\partial}{\partial b} B\left(p+1,\frac{(b+1)n}{a}\right)\\
			=&\frac{n}{a}B\left(p+1,\frac{(b+1)n}{a}\right)\left[\psi \left(\frac{(b+1)n}{a}\right)-\psi \left(p+1+\frac{(b+1)n}{a}\right)\right]
		\end{align*}
		and
		\begin{align*}
			&\frac{\partial}{\partial b} B\left(p+1,\frac{(b+1)n}{b}\right)\\
			=&-\frac{n}{b^2}B\left(p+1,\frac{(b+1)n}{b}\right)\left[\psi \left(\frac{(b+1)n}{b}\right)-\psi \left(p+1+\frac{(b+1)n}{b}\right)\right].
		\end{align*}
		Hence
		\begin{align*}
			\frac{\partial F}{\partial b}=& \left(\frac{np+n}{n+p}\right)a^{-\frac{np+n}{n+p}}b^{\frac{np+n}{n+p}-1}B\left(p+1, \frac{(b+1)n}{a}\right)\\
			&+\frac{p}{n+p} \left(\frac{b}{a}\right)^{\frac{np+n}{n+p}}\left(a+1\right)^{-\frac{p}{n+p}}\left(b+1\right)^{\frac{p}{n+p}-1}B\left(p+1, \frac{(b+1)n}{a}\right)\\
			&+ \left(\frac{b}{a}\right)^{\frac{np+n}{n+p}}\left(\frac{b+1}{a+1}\right)^{\frac{p}{n+p}}\frac{n}{a}B\left(p+1,\frac{(b+1)n}{a}\right)\times\\
			& \quad \quad \times \left[\psi \left(\frac{(b+1)n}{a}\right)-\psi \left(p+1+\frac{(b+1)n}{a}\right)\right]\\
			&+\frac{n^2}{b^2(n+p)}B\left(p+1, \frac{(a+1)n}{a}\right)^{\frac{p}{n+p}}B\left(p+1,\frac{(b+1)n}{b}\right)^{\frac{n}{n+p}}\times\\
			& \quad \quad \times \left[\psi \left(\frac{(b+1)n}{b}\right)-\psi \left(p+1+\frac{(b+1)n}{b}\right)\right].
		\end{align*}
		This implies that $\frac{\partial F}{\partial b}(1,1)$ is equal to
		\begin{align*}
			&B\left(p+1,2n\right)\left[\frac{np+n}{n+p}+\frac{p}{2(n+p)}+\left(n+\frac{n^2}{n+p}\right)\left(\psi(2n)-\psi(2n+p+1)\right)\right]\\
			=&\frac{(2n^2+np)B(p+1,2n)}{n+p}\left[\frac{2np+2n+p}{2(2n^2+np)}+\psi (2n)-\psi(2n+p+1)\right]\\
			=&\frac{(2n^2+np)B(p+1,2n)}{n+p}\left[\frac{1}{2n}+\frac{p}{2n+p}+\psi (2n)-\psi(2n+p+1)\right].
		\end{align*}
		Hence, by Lemma \ref{lem}, we obtain that $\frac{\partial F}{\partial b}(1,1)\neq 0$ for any $p>0, p\neq 1$.   
	\end{example}
	
	\section*{Acknowledgments}
	The first named author is supported by the MOE grant MOE-T2EP20120-0010. The authors would like to thank Professor Pham Hoang Hiep and Associate Professor Do Hoang Son for very useful discussions.
	
\end{document}